\documentclass[reqno,11pt]{amsart}
\usepackage{amsmath,amssymb,amsthm,amsfonts}
\usepackage{mathrsfs}
\usepackage{bbm}
\usepackage{hyperref}

\usepackage{color}

\newtheorem{lemma}{Lemma}[section]
\newtheorem{theorem}{Theorem}[section]
\newtheorem{definition}{Definition}[section]
\newtheorem{proposition}{Proposition}[section]

\numberwithin{equation}{section}

\newcommand{\R}{\mathbb{R}}

\renewcommand{\S}{\mathbb{S}}

\newcommand{\na}{\nabla}

\newcommand{\ga}{\gamma}
\newcommand{\om}{\omega}
\newcommand{\Om}{\Omega}
\newcommand{\la}{\lambda}
\newcommand{\de}{\delta}

\newcommand{\pa}{\partial}

\newcommand{\eps}{\epsilon}

\newcommand{\vps}{\varepsilon}

\begin{document}
\title[Boltzmann equation with weakly inhomogeneous data in bounded domain]{The Boltzmann equation with weakly inhomogeneous data in bounded domain}

\author[Y. Guo]{Yan Guo}
\address[YG]{Division of Applied Mathematics, Brown University, Providence 02912, USA}
\email{${\rm Yan}_-{\rm Guo}$@brown.edu}

\author[S.-Q. Liu]{Shuangqian Liu}
\address[SQL]{Department of Mathematics, Jinan University, Guangzhou 510632, P.R.~China}
\email{tsqliu@jnu.edu.cn}

\begin{abstract}
This paper is concerned with the Boltzmann equation with specular reflection boundary condition. We construct a unique  global solution and obtain its large time asymptotic behavior in the case that the initial data is close enough to a radially symmetric homogeneous datum. The result extends the case of Cauchy problem considered by Arkeryd-Esposito-Pulvirenti [Comm. Math. Phys. 111(3): 393-407 (1987)] to the specular reflection boundary value problem in bounded domain.
\end{abstract}



\maketitle

\thispagestyle{empty}
\tableofcontents

\section{Introduction}
\subsection{The problem}
If a dilute gas is contained in a bounded region with completely smooth boundary surface and the gas molecules collide the surface
elastically,
the motion of those gas particles can be modeled by the following initial value problem for the Botlzmann
equation with specular reflection boundary condition
\begin{eqnarray}\label{be}
\left\{\begin{array}{lll}
\begin{split}
&\pa_tF+v\cdot\na_x F=Q(F,F),\ t>0, \ x\in\Omega,\ v\in\R^3,\\
&F(0,x,v)=F_0(x,v),\ x\in\Omega,\ v\in\R^3,\\
&F(t,x,v)|_{n(x)\cdot v<0}=F(t,x,R_xv),\ R_xv=v-2(v\cdot n(x))n(x),\ \ t\geq0,\ x\in\Omega,\ v\in\R^3.
\end{split}
\end{array}\right.
\end{eqnarray}
Here, $F(t,x,v)\geq0$ denotes the density distribution function of the gas particles at time $t\geq0$, position $x\in\Omega$, and velocity $v\in\R^3$,
$\Omega$ is a bounded domain in $\R^3$, $n(x)$ is the outward pointing unit norm vector at boundary $x\in\pa\Omega$.
The Boltzmann collision operator $Q(\cdot,\cdot)$ for hard sphere model is given as the following non-symmetric form
\begin{equation*}\label{n.op}
\begin{split}
Q(F_1,F_2)=&\int_{\R^3\times\S^2}|(u-v)\cdot\omega|[F_1(u')F_2(v')-F_1(u)F_2(v)]dud\omega\\
=&Q_{\textrm{gain}}(F_1,F_2)-Q_{\textrm{loss}}(F_1,F_2),
\end{split}
\end{equation*}
here, $(u,v)$ and $(u', v')$ stand for the velocities of the particles before and after the collision and satisfy
\begin{eqnarray*}\label{v.re}
\left\{\begin{array}{lll}
\begin{split}
&v'=v+[(u-v)\cdot\omega]\om,\ \ u'=u-[(u-v)\cdot\omega]\om,\\
&|u|^2+|v|^2=|u'|^2+|v'|^2.\end{split}\end{array}\right.
\end{eqnarray*}

In the present paper, we study the existence of unique global solution of \eqref{be} and its time asymptotic behaviors when the initial particle distribution $F_0(x,v)$ is a small perturbation around a spatially homogeneous radially symmetric datum. Specifically speaking, let $G_0(v)=G(|v|)$ be a spatially homogeneous datum and set $F_0=G_0+f_0$. We hope to construct unique global solution of \eqref{be} in the case that $f_0$ is small in a suitable sense.

It is known from \cite{Car-1933} that there exists a unique global radially symmetric solution $G(t,v)=G(t,|v|)$ to the Cauchy problem for the spatially homogeneous Boltzmann equation
\begin{eqnarray}\label{hbe}
\left\{\begin{array}{ll}
\begin{split}
&\pa_tG=Q(G,G),\ t>0,\ v\in\R^3,\\
&G(0,v)=G_0(v)=G_0(|v|),\ v\in\R^3.
\end{split}\end{array}\right.
\end{eqnarray}
Based on such a seminal work by Carleman, the aforementioned problem was first considered by Arkeryd-Esposito-Pulvirentic \cite{AEP-87}, in which they constructed the global existence for the Cauchy problem for the inhomogeneous Boltzmann equation with weakly inhomogeneous datum, namely the initial distribution is sufficiently close to the spatially homogeneous datum in \eqref{hbe}.
Since then, very few results were known regarding the inhomogeneous Botlzmann equation with weakly inhomogeneous datum, although there are extensive investigations on the homogeneous Boltzmann equation \eqref{hbe}, see \cite{Ark-83,Ark-88,GPV-09,Mo-2006} and reference therein.



\subsection{Domain and Characteristics}
Throughout this paper, $\Omega $ is a connected and bounded domain in $
\R^3$, which is defined as
$
\Omega=\{x\in \R^3~|~\xi(x)<0\}
$
with $\xi(x)$ being a smooth function. Let $\na\xi(x)\neq0$ at boundary $\xi(x)=0$, the outward pointing unit normal vector at every point $x\in\pa\Omega$ is given by
$$
n(x)=\frac{\na\xi(x)}{|\na\xi(x)|}.
$$
If there exists $c_\xi>0$, for any $\zeta=(\zeta^1,\zeta^2,\zeta^3) \in \R^3$ such that
\begin{equation}\label{scon}
\pa_{ij}\xi(x)\zeta^i\zeta^j\geq c_\xi|\zeta|^2,
\end{equation}
then we say $\Om$ is {\it strictly convex}.
To take care of the specular reflection boundary condition $\eqref{be}_3$ (the third equation of \eqref{be}), we may also require  $\Omega $ has a {\it rotational symmetry}, that is, there are vectors $x_{0}$
and $\varpi ,$ such that for all $x\in \partial \Omega $
\begin{equation}
\{(x-x_{0})\times \varpi \}\cdot n(x)\equiv 0.  \label{axis}
\end{equation}

For the sake of convenience, the phase boundary in the phase space $\Omega \times \R
^{3} $ is denoted by $\gamma =\partial \Omega \times \R^{3}$, and we further split it into
the following three kinds:
\begin{eqnarray*}
\textrm{outgoing boundary}: \gamma _{+} &=&\{(x,v)\in \partial \Omega \times \R^{3}\ :\
n(x)\cdot v> 0\}, \\
\textrm{incoming boundary}: \gamma _{-} &=&\{(x,v)\in \partial \Omega \times \R^{3}\ :\
n(x)\cdot v< 0\}, \\
\textrm{grazing boundary}: \gamma _{0} &=&\{(x,v)\in \partial \Omega \times \R^{3}\ :\
n(x)\cdot v=0\}.
\end{eqnarray*}%
Given $(t,x,v)$, we let $[X(s),V(s)]$ satisfy
\begin{equation*}
\frac{dX(s)}{ds}=V(s),\ \ \frac{dV(s)}{ds}=0,
\end{equation*}
with the initial data $[X(t;t,x,v),V(t;t,x,v)]=[x,v]$. Then $[X(s;t,x,v),V(s;t,x,v)]$$=[x-(t-s)v,v]$$=[X(s),V(s)]$, which is called as the backward characteristic trajectory for the Boltzmann equation
$\eqref{be}_1$.

For $%
(x,v)\in \Om\times \R^{3}$, the \textit{backward exit time} $t_{\mathbf{b}}(x,v)>0$ is defined to be the first moment at which the backward characteristic line
$[X(s;0,x,v),V(s;0,x,v)]$ emerges from $\pa\Om$:
\begin{equation*}\label{b.t}
t_{\mathbf{b}}(x,v)=\inf \{\ t> 0:x-tv\notin \partial \Omega \},
\end{equation*}
and we also define $x_{\mathbf{b}}(x,v)=x-t_{\mathbf{b}}(x,v)v\in \partial \Omega $. Note that for any $(x,v)$, we use $t_{\mathbf{b}}(x,v)$
whenever it is well-defined.

\subsection{Main results}
Our main goal in the paper is to prove the global existence and obtain the time asymptotic behaviors of the initial boundary value problem \eqref{be}
on the condition that the initial datum $F_0$ is sufficiently close to a homogeneous initial datum $G_0(v)$ in an appropriate sense.
Let us first assume the homogeneous data $G_0$ shares the same mass and energy conservation laws with $F_0$:
\begin{eqnarray}\label{con.id}
\int_{\R^3}G_0\left(\begin{array}{rll}
&1\\
&v^2
\end{array}\right)dv=\int_{\R^3}F_0\left(\begin{array}{rll}
&1\\
&v^2
\end{array}\right)dv=\left(\begin{array}{rll}
&1\\
&3
\end{array}\right).
\end{eqnarray}
Moreover, if the domain $\Omega $ has any axis of rotational symmetry
\eqref{axis}, then we further assume the corresponding conservation of
angular momentum is valid for the initial data $F_0$
\begin{equation}
\int_{\Omega \times \R^{3}}\{(x-x_{0})\times \varpi \}\cdot vF_0(t,x,v)%
dxdv=0.  \label{axiscon}
\end{equation}
Through the paper, the global Maxwellian $\mu$ is defined as $(2\pi)^{-3/2}e^{-\frac{|v|^2}{2}}$.
Denote $\langle v\rangle=\sqrt{1+|v|^2}$ and we introduce a velocity weight $w_l=\langle v\rangle^l$ with $l\geq0.$
We use $\Vert \,\cdot \,\Vert _{\infty}$ to denote the $L^{\infty}(\Omega
\times \R_v^{3})-$norm or $L^{\infty}(\R_v^{3})-$norm.
Moreover,
$(\cdot,\cdot)$ denotes the $L^{2}$ inner product in
$\Omega\times {\R}^{3}$  with
the $L^{2}$ norm $\|\cdot\|_2$.

We now state our main results in the following theorem.
\begin{theorem}\label{mth}
Assume that $\xi $ is both strictly convex \eqref{scon} and analytic. Suppose \eqref{con.id} is valid. In the case of $\Omega $ has any rotational symmetry \eqref{axis}, we
further require the corresponding angular momentum \eqref{axiscon} holds. There exists $l_0 >6$ such that if $\|w_{l_0}G_0\|_\infty<\infty$, then there exists $\vps_0>0$  such that for $F_{0}(x,v)\geq 0$ and $\|w_{l_0}[F_{0}-G_0]\|_{\infty }\leq \vps_0 ,$ there
exists a unique solution $F(t,x,v)\geq 0$ to the Boltzmann equation \eqref{be} satisfying $\sup\limits_{0\leq t\leq \infty}\|w_{l_0}F(t)\|_\infty<\infty$. Moreover,
\begin{equation*}
\|w_{l_0}[F(t)-\mu]\|_{\infty }\rightarrow 0,\ \textrm{as}\ t\rightarrow\infty.
\end{equation*}%
In addition, if $F_{0}(x,v)$ is continuous except on the set $\gamma _{0}$, 
then $F(t,x,v)$ is continuous in $[0,\infty )\times \{\bar{\Omega}\times
\R^{3}\setminus \gamma _{0}\}.$
\end{theorem}

Let us now give a brief review on the existing results devoted to the Boltzmann equation with specular reflection boundary condition. For the initial boundary value problem \eqref{be}, if the initial data $F_0$ is allowed to be sufficiently close to the the absolute Maxwellian $\mu$, many advances have occurred over the past few decades. Shizuta-Asano \cite{SA-77} announced that the solutions to the Boltzmann equation near a Maxwellian would tend exponentially to the same equilibrium in a smooth bounded convex domain, although there is no complete rigorous proof in the paper.
Ukai \cite{Ukai-86} presented a general idea for proving the existence and time convergence to a global Maxwellian for the initial boundary value problem with hard potential. Golse-Perthame-Sulem \cite{GPS} investigated the stationary boundary layers of the Boltzmann equation around the global Maxellian in half spatial space in the case of hard spheres model.  By applying a basic energy method, Yang-Zhao \cite{YZ} proved the stability of the rarefaction waves for the one dimensional Boltzmann equation in half space. Based on the a priori assumption that some strong Sobolev estimates can be
obtained, Desvillettes and Villani \cite{Des-90, DV-05, V1, V2} established an almost exponential decay rate for
Boltzmann solutions with large amplitude for general collision kernels and general
boundary conditions. However,
unlike the Cauchy problem for the Boltzmann equation in the whole space or on the torus, the solutions of the Boltzamnn equation with many of physical boundary conditions may create singularities in general domains \cite{Kim}, this is one of the major mathematical obstacles to
study the nonlinear Boltzmann equation with boundary conditions in the Sobolev space. Recently, the first author of the paper developed an  $L^2-L^\infty$ theory to establish the time decay and continuity of the unique global solution of the Boltzmann equation with four basic boundary conditions: in flow, bounce back, specular reflection and diffuse reflection \cite{Guo-2010}. The result is then extended to the soft potential case by the second author of the paper and Yang \cite{LY-2016}. More recently, the $W^{1,p}$ $(1<p<2)$ regularity for the Botlzmann equation in general classes of bounded domain were further proved \cite{GKTT-12} . There is also a huge number of literatures concerning the mathematical studies for the other kinds of boundary conditions, we refer to \cite{AC-93,Cer-92,CIP,EGM,EGKM-13,GKTT-14,Ha-92,LY-im,LY-IBP-07,LY-DCDS,Yu} and references therein. Nevertheless, if the initial data $F_0(x,v)$ is away from a global Maxwellian, we are not aware of any results for the initial boundary value problem \eqref{be}. Our results in Theorem \ref{mth}
appears to be the first one devoted to the study of this problem.

Following the ideas developed in \cite{AEP-87}, the proof for Theorem \ref{mth} is divided into three steps: Step 1, we show that there exists a unique inhomogeneous solution $F$ which is close enough to a homogeneous solution determined by \eqref{hbe}, at least for a finite time interval which is inversely proportional to the size of the initial amplitude; Step 2, with the aid of an $L^2-L^\infty$ argument, we prove that $F$ tends to Maxwellian equilibrium exponentially provided that the initial discrepancy is sufficiently small in a suitable sense; Step 3, by combing the first step and the fact that the homogeneous solution $G$ converges to the global maxwellian $\mu$ as time goes to infinity, we find a time at which the
inhomogeneous solution $F$ enters the small neighbourhood of the global maxwellian $\mu$, this together with the second step enables one to extend the local solution to the global one.

The organization of the paper is arranged as follows: In section \ref{pre}, we present some known results which will be used in the subsequent sections. In
Section \ref{lochom}, we establish a unique local solution to the Boltzmann equation with specular reflection boundary condition, in particular, such a local solution is close enough to a homogeneous solution for the corresponding
homogeneous Boltzmann equation. Section \ref{smsol} is devoted to the existence of a global small solution to the initial boundary value problem of the Boltzmann equation around an absolute Maxwellian.
The proof of Theorem \ref{mth} is concluded in the Section \ref{proof}.

{\it Notations:}
We now list some notations used in the paper.
 Throughout this paper,  $C$ denotes some generic positive (generally large) constant and $\la,\la_1,\la_2$ as well as $\la_0$ denote some generic positive (generally small) constants, where $C$ may take different values in different places. $D\lesssim E$ means that  there is a generic constant $C>0$
such that $D\leq CE$. $D\sim E$
means $D\lesssim E$ and $E\lesssim D$. For brevity,
 We also denote $f_{\pm}=f|_{\gamma_{\pm}}=f\mathbf{1}_{\gamma_{\pm}}$.

\section{Preliminary}\label{pre}
In this section, we collect some basic results and significant estimates which will be used in the latter proof. The first one is concerned with the global existence of \eqref{hbe}.
\begin{proposition}\label{Gex}\cite[pp.133, pp146]{Car-1933}Assume $G_0(v)=G_0(|v|)\geq0$ satisfies \eqref{con.id}, if $\|w_{l_1}G_0\|_\infty<\infty$ for $l_1>6$, then the Cauchy problem \eqref{hbe} admits a unique global radially symmetric solution $G(t,v)\geq0$ satisfying
\begin{equation}\label{Ges}
\sup\limits_{t\geq 0}\|w_{l_1} G(t)\|_\infty\leq C_0,
\end{equation}
for some $C_0>0.$ Moreover $G(t,v)$ is continuous if $G_0(v)$ is continuous, and there exists $l_1>l_0>6$ such that
\begin{equation*}\label{Gla}
\sup\limits_{t\rightarrow\infty }\|w_{l_0} [G(t)-\mu]\|_\infty=0.
\end{equation*}
\end{proposition}
Recall that $Q_{\textrm{loss}}(F_1,F_2)$ can be rewritten as $F_2RF_1$
with
$$RF_1=\int_{\R^3\times \S^2}|(u-v)\cdot \omega|F_1(u)dud\omega.$$
The following lemma states that $RG$ enjoys a nice lower bound.
\begin{lemma}\cite[pp.99, pp.121]{Car-1933}Let $G(v)$ be a unique global solution constructed in Proposition \ref{Gex}, then there exists $\nu_0>0$ such that
\begin{equation}\label{Glbb}
RG\geq \nu_0\langle v\rangle.
\end{equation}
\end{lemma}
Finally, we address a crucial weighed estimates on the bilinear Boltzmann operator $Q(\cdot,\cdot)$.
\begin{lemma}\label{op.es.lem}For $l>6$, it holds that
\begin{equation}\label{op.es}
|w_l Q(F_1,F_2)|\leq C\|w_l F_1\|_\infty\|w_l F_2\|_\infty\langle v\rangle.
\end{equation}
Furthermore, there exists a suitably small $\eps>0$ depending on $l$ such that
\begin{equation}\label{op.es2}
\begin{split}
|w_l Q_{\textrm{gain}}(F_1,F_2)|&+|w_l Q_{\textrm{loss}}(F_1,F_2)|+|w_l Q_{\textrm{gain}}(F_2,F_1)|\\
\leq&
 \|w_l F_1\|_\infty\left\{C\|w_{l+1} F_2\|_\infty+\eps\|w_lF_2\|_\infty\langle v\rangle\right\}.
\end{split}
\end{equation}
\end{lemma}
\begin{proof}
The proof for \eqref{op.es} and \eqref{op.es2} is the same as that of (3.7) and (3.8) in \cite[pp.397]{AEP-87}, respectively, we omit the details for brevity.
\end{proof}

\section{Local existence for the weakly inhomogeneous data}\label{lochom}
In this section, we will show that \eqref{be} admits a unique local-in-time solution around the spatially homogeneous solution $G(t,v)$
of \eqref{hbe}.

Denote $f=F-G$, then $f$ satisfies
\begin{eqnarray}\label{feq}
\left\{
\begin{array}{ll}
\begin{split}
&\pa_tf+v\cdot \na_xf+fRG=Q(f,G)+Q_{\textrm{gain}}(G,f)+Q(f,f),\ t>0, \ x\in\Omega,\ v\in\R^3,\\
&f|_{t=0}=f_0(x,v)=F_0-G_0,\ \ f(t,x,v)|_{_-}=f(t,x,R_xv),\ t\geq0, \ x\in\Omega,\ v\in\R^3,
\end{split}
\end{array}\right.
\end{eqnarray}
where the fact that $G(v)=G(|v|)=G(|R_xv|)$ was used.
To resolve the initial boundary value problem \eqref{feq}, it is necessary to introduce the following specular reflection cycles:
\begin{definition}
Let $\Omega$
be convex \eqref{scon}. Fix any
point $(t,x,v)\notin \gamma _{0}\cap \gamma _{-},$ and define $%
(t_{0},x_{0},v_{0})=(t,x,v)$, and for $k\geq 1$
\begin{equation}
(t_{k+1},x_{k+1},v_{k+1})=(t_{k}-t_{\mathbf{b}}
(t_{k},x_{k},v_{k}),x_{%
\mathbf{b}}(x_{k},v_{k}),R_{x_{k+1}}v_{k}),  \label{specularcycle}
\end{equation}%
where $R_{x_{k+1}}v_{k}=v_{k}-2(v_{k}\cdot n(x_{k+1}))n(x_{k+1}).$ And we
define the specular back-time cycle
\begin{equation*}
X_{\mathbf{cl}}^{{}}( s )\equiv \sum_{k=1}\mathbf{1}%
_{[t_{k+1},t_{k})}( s )\left\{ x_{k}+v_{k}( s -t_{k})\right\} ,\text{ \ \ }V_{%
\mathbf{cl}}^{{}}( s )\equiv \sum_{k=1}\mathbf{1}%
_{[t_{k+1},t_{k})}( s )v_{k}.  \label{backtimecycle}
\end{equation*}
\end{definition}

Our main result in this section is the following:
\begin{proposition}\label{loc.ex}
Given $t>0$, there exists constant $C_1>0$
such that if
\begin{equation}\label{aps}
\|w_lf_0\|_\infty\leq (16C_1)^{-1}e^{-(2C_1\ln 4)t},
\end{equation}
for $l>6$,
 then there exists a unique solution up to time $t>0$
of \eqref{feq} satisfying
\begin{equation}\label{tbd}
\|w_lf(t)\|_\infty\leq \|w_lf_0\|_\infty e^{(2C_1\ln 4)t}.
\end{equation}
Moreover, if $f_{0}(x,v)$ is continuous away from the set $\gamma _{0}$,
then $f(s,x,v)$ is continuous in $[0,t ]\times \{\bar{\Omega}\times
\R^{3}\setminus \gamma _{0}\}.$

\end{proposition}
\begin{proof}
Let us first deduce the a priori estimate \eqref{tbd} under the assumption \eqref{aps}. Suppose $f(t,x,v)$ is a solution to the initial boundary value problem \eqref{feq} on the
time interval $[0,t]$. Take $(t,x,v)\notin [0,\infty)\times(\ga_0\cup\ga_-)$ and recall the definition $t_k$ in \eqref{specularcycle}, to see that there exists a finite integer $m>0$ such that $t_{m+1}\leq 0<t_m$, cf. \cite[pp.768]{Guo-2010}. From \eqref{feq}, one thus has for $(t,x,v)\notin [0,\infty)\times(\ga_0\cup\ga_-)$
\begin{equation}\label{f.ep}
\begin{split}
f(t,x,v)=&e^{-\int_0^tRG (s)ds}f_0+\sum\limits_{k=0}^{m-1}\int_{t_{k+1}}^{t_k}e^{-\int_{s}^{t}RG (\tau)d\tau}\phi_k(s)ds\\
&+\int_{0}^{t_m}e^{-\int_{s}^{t}RG (\tau)d\tau}\phi_m(s)ds,
\end{split}
\end{equation}
here $\phi_k(s)$ is defined as $\phi_k(s)=\left\{Q(f,G)+Q_{\textrm{gain}}(G,f)+Q(f,f)\right\}(s,x_k+(s-t_{k})v_k,v_k).$ On the one hand,
\eqref{f.ep}
implies
\begin{equation}\label{f.ep1}
\begin{split}
|w_lf|\leq &e^{-\int_0^tRG (s)ds}|w_lf_0|+\sum\limits_{k=0}^{m-1}\int_{t_{k+1}}^{t_k}e^{-\int_{s}^{t}RG (\tau)d\tau}|w_l\phi_k(s)|ds\\
&+\int_{0}^{t_m}e^{-\int_{s}^{t}RG (\tau)d\tau}|w_l\phi_m(s)|ds.
\end{split}
\end{equation}
On the other hand, Lemma \ref{op.es.lem} and \eqref{Ges} in Proposition \ref{Gex} give rise to
\begin{equation}\label{phibd}
\begin{split}
|w_l\phi_k(s)|\leq C_l\|w_lf(s)\|_\infty+\eps\|w_lf(s)\|_\infty\langle v\rangle+C_l\|w_lf(s)\|^2_\infty\langle v\rangle,
\end{split}
\end{equation}
for all $0\leq k\leq m.$
Plugging \eqref{f.ep1} and \eqref{phibd} into \eqref{f.ep} and applying \eqref{Glbb} 
 we arrive at
\begin{equation*}\label{fes}
\begin{split}
\|w_lf(t)\|_{\infty}\leq \|w_lf_0\|_{\infty}+C_lt\sup\limits_{0\leq s\leq t}\|w_lf(s)\|_\infty+\frac{\eps}{\nu_0}\sup\limits_{0\leq s\leq t}\|w_lf(s)\|_\infty+\frac{C_l}{\nu_0}\sup\limits_{0\leq s\leq t}\|w_lf(s)\|^2_\infty.
\end{split}
\end{equation*}
Next, taking $C_1=\max\{C_l/\nu_0,C_l+1\}$, we further have by letting $\frac{\eps}{\nu_0}\leq t_\ast$
\begin{equation}\label{X}
\begin{split}
\|w_lf_0\|_{\infty}-(1-t_\ast C_1)\sup\limits_{0\leq s\leq t_\ast}\|w_lf(s)\|_{\infty}+C_1\sup\limits_{0\leq s\leq t_\ast}\|w_lf(s)\|^2_{\infty}\geq 0.
\end{split}
\end{equation}
Hence, for $t_\ast=1/(2C_1)$, we get from \eqref{X} that
\begin{equation}\label{X1}
\begin{split}
\sup\limits_{0\leq s\leq t_\ast}\|w_lf(s)\|_{\infty}\leq 4\|w_lf_0\|_{\infty},
\end{split}
\end{equation}
provided that $\|w_lf_0\|_{\infty}\leq 1/(16C_1)$. Since $[\frac{t}{t_\ast}]t_\ast\leq t\leq ([\frac{t}{t_\ast}]+1)t_\ast$ ($[x]$ denotes the integer part of $x$), using \eqref{aps} and iterating \eqref{X1} $[\frac{t}{t_\ast}]$ times lead us to
\begin{equation*}
\|w_lf(t)\|_\infty\leq \|w_lf_0\|_\infty e^{[\frac{t}{t_\ast}]\ln 4}\leq \|w_lf_0\|_\infty e^{(2C_1\ln 4)t}.
\end{equation*}
We are now in a position to prove the existence of unique solution of \eqref{feq}, to do this, we first design the following iteration approximation sequence 
\begin{eqnarray*}\label{it1}
\left\{
\begin{array}{ll}
\begin{split}
&\pa_tf^{n+1}+v\cdot \na_xf^{n+1}+f^{n+1}RG\\&\qquad\qquad=Q(f^{n},G)+Q_{\textrm{gain}}(G,f^{n})+Q(f^{n},f^{n}),\ t>0, \ x\in\Omega,\ v\in\R^3,\\
&f^{n+1}|_{t=0}=f_0(x,v)=F_0-G_0,\ \ f^{n+1}(t,x,v)|_{-}=f^{n+1}(t,x,R_xv),\ t\geq0, \ x\in\Omega,\ v\in\R^3,
\end{split}
\end{array}\right.
\end{eqnarray*}
starting with $f^0(t,x,v)=f_0(x,v)$. It is quite routine to show that $\{f^n\}_{n=0}^\infty$ is well-defined and enjoys the upper bound
\begin{equation}\label{fnbd}
\sup\limits_{n}\sup_{0\leq t\leq \tau_\ast}\|w_lf^n(t)\|_\infty\leq C\|w_lf_0\|_\infty,
\end{equation}
for some small $\tau_\ast>0$.
In what follows, we will verify that such a sequence is convergent in the weighted $L^\infty$ space for a small $\tau_\ast>0$.
For convenience, we denote $\phi^n$ by $Q(f^{n},G)+Q_{\textrm{gain}}(G,f^{n})+Q(f^{n},f^{n})$, and we also use the similar notation $\phi_k^n$ as $\phi_k$. Let $g^{n+1}=f^{n+1}-f^{n}$, then one sees that $g^{n+1}$ satisfies
\begin{eqnarray*}\label{it2}
\left\{
\begin{array}{ll}
\begin{split}
&\pa_tg^{n+1}+v\cdot \na_xg^{n+1}+g^{n+1}RG=\phi^n-\phi^{n-1},\ t>0, \ x\in\Omega,\ v\in\R^3,\\
&g^{n+1}|_{t=0}=0,\ \ g^{n+1}(t,x,v)|_{-}=g^{n+1}(t,x,R_xv),\ t\geq0, \ x\in\Omega,\ v\in\R^3,
\end{split}
\end{array}\right.
\end{eqnarray*}
from which, it follows for $(t,x,v)\notin [0,\infty)\times(\ga_0\cup\ga_-)$
\begin{equation}\label{gn.ep}
\begin{split}
g^{n+1}(t,x,v)=&\sum\limits_{k=0}^{m-1}\int_{t_{k+1}}^{t_k}e^{-\int_{s}^{t}RG (\tau)d\tau}[\phi^n_k-\phi^{n-1}_k](s)ds
\\&+\int_{0}^{t_m}e^{-\int_{s}^{t}RG (\tau)d\tau}[\phi^n_m-\phi^{n-1}_m](s)ds.
\end{split}
\end{equation}
Notice that
\begin{equation}\label{gnnop}
\begin{split}
[\phi^n_k-\phi^{n-1}_k]=Q(g_k^{n},G_k)+Q_{\textrm{gain}}(G_k,g_k^{n})+Q(g_k^{n},f_k^{n})+Q(f_k^{n-1},g_k^{n}),
\end{split}
\end{equation}
where $G_k=G(v_k)$. Letting $\|w_lf_0\|_\infty$ be sufficiently small, for a small $\tau_\ast>0$, we get from Lemma \ref{op.es.lem}, \eqref{fnbd}, \eqref{gn.ep} and \eqref{gnnop} that
\begin{equation}\label{gnct}
\begin{split}
\sup\limits_{0\leq t\leq \tau_\ast}\|w_lg^{n+1}(t)\|_\infty\leq c(\tau_\ast,\|w_lf_0\|_\infty,\eps_0)\sup\limits_{0\leq t\leq \tau_\ast}\|w_lg^{n}(t)\|_\infty,
\end{split}
\end{equation}
where $0<c(\tau_\ast,\|w_lf_0\|_\infty,\eps_0)<1.$ Therefore the iteration approximation sequence $\{f^n\}_{n=0}^\infty$ converges in a small time interval $[0,\tau_\ast]$ provided both $\|w_lf_0\|_\infty$ and $\eps_0$ are small enough. In fact the procedure for obtaining \eqref{gnct} can be iterated to get the convergence up to time $t$ due to \eqref{aps} and \eqref{tbd}. The uniqueness and positivity of the solution follows trivially. In addition, since we have $L^\infty$ convergence, as \cite[pp.804]{Guo-2010}, one can further deduce that $f$ is continuous away from $\ga_0$ when $\Omega$ is strictly convex, this ends up the proof of Proposition \ref{loc.ex}.

\end{proof}

\section{Exponential decay for the small data around a global Maxwellian}\label{smsol}
For the Boltzmann equation with boundary condition, if the initial data is assumed to be a small perturbation around a global Maxwellian, as
mentioned in the introduction, the existence, uniqueness and regularity as well as their time decay toward the absolute Maxwllian have been
intensively studied, see \cite{Guo-2010,BG-2015,LY-2016}. Generally speaking, there are two basic kinds of perturbative regime around a global equilibrium $\mu$:
$F=\mu+h$ and $F=\mu+\sqrt{\mu}h$. One of the advantages of the former decomposition is that it allows much more weaker velocity
weight than the latter one when one aims to obtain the global existence and long time behaviors of the solutions with the aid of the weighted $L^\infty$ approach, cf. \cite{Guo-2010, BG-2015}.

Using the splitting $F=\mu+h$, we rewrite \eqref{be} as
\begin{eqnarray}\label{h.eq}
\left\{\begin{array}{lll}
\begin{split}
&\pa_th+v\cdot\na_x h+Lh=Q(h,h),\ t>0, \ x\in\Omega,\ v\in\R^3,\\
&h(0,x,v)=h_0(x,v)=F_0-\mu,\ x\in\Omega,\ v\in\R^3,\\
&h(t,x,v)|_{n(x)\cdot v<0}=h(t,x,R_xv),\ t\geq0,\ x\in\Omega,\ v\in\R^3,
\end{split}
\end{array}\right.
\end{eqnarray}
here the linear operator $L$ is defined as
\begin{equation*}\label{L.def}
Lh=-\left\{Q(\mu,h)+Q(h,\mu)\right\}=\nu h-Kh,
\end{equation*}
with $\nu=Q_{\textrm{loss}}(\mu,1)\sim \langle v\rangle$ and $Kh=Q(h,\mu)+Q_{\textrm{gain}}(\mu,h)$.

Our main results in this section is the following:
\begin{proposition}\label{h.exist}
Assume that $\xi $ is both strictly convex \eqref{scon} and analytic,
and the mass and energy \eqref{con.id} are conserved for $h_{0}+\mu$
. In the case of $\Omega $ has any rotational symmetry \eqref{axis}, we
further require the corresponding angular momentum \eqref{axiscon} is conserved
for $h_{0}+\mu$.
There exists $\de_0>0$, such that if $\|w_lh_0\|_{\infty}\leq\de_0$ with $l>6$, then there exists a unique global solution
$h(t,x,v)$ to the initial boundary value problem \eqref{h.eq}. Moreover, there exists $\la>0$ such that
\begin{equation}\label{hdecay}
\|w_lh(t)\|_{\infty}\leq Ce^{-\la t}\|w_lh_0\|_{\infty}.
\end{equation}
In addition, if $h_{0}(x,v)$ is continuous away from the set $\gamma _{0}$, 
then $h(t,x,v)$ is continuous in $[0,\infty )\times \{\bar{\Omega}\times
\R^{3}\setminus \gamma _{0}\}.$

\end{proposition}

To prove Proposition \ref{h.exist}, the key point is to find a decomposition of the perturbed Boltzmann linear operator $L$
into
\begin{equation*}\label{L.dec}
L=\nu-\chi^c_N K-\chi_NK,
\end{equation*}
where for $N>0$, $\chi_N=\left\{\begin{array}{ll}
&1,\ \ \textrm{if}\ \ |v|< N,\\
&0,\ \ \textrm{ortherwise}
\end{array}\right.$, and $\chi^c_N=1-\chi_N$. It is important to point out that $\chi^c_N K$ is small compared to $\nu$ and $\chi_NK$ has a smooth effect.
This idea comes from Caflisch \cite{Ca-1980} and Arkeryd-Esposito-Pulvirenti \cite{AEP-87}, for its recent application in the study of the Boltzmann equation in weighted $L_v^1L_x^\infty$ space, we refer to \cite{BG-2015, BD-2016}.
More precisely, we decompose $h=h_1+\sqrt{\mu}h_2$, then we shall construct $(h_1,h_2)$ solution to the following initial boundary value problem
\begin{eqnarray}\label{h1}
\left\{\begin{array}{ll}
&\pa_t h_1+v\cdot\na_xh_1+\nu h_1=\chi_N^c Kh_1+Q(h_1+\sqrt{\mu}h_2,h_1+\sqrt{\mu}h_2),\\[2mm]
&h_1(0,x,v)=h_0(x,v),\ \ h_1(t,x,v)|_-=h_1(t,x,R_xv),
\end{array}\right.
\end{eqnarray}
\begin{eqnarray}\label{h2}
\left\{\begin{array}{ll}
&\pa_t h_2+v\cdot\na_xh_2+\widetilde{L} h_2=\mu^{-1/2}\chi_N Kh_1,\\[2mm]
&h_2(0,x,v)=0,\ \ h_2(t,x,v)|_-=h_2(t,x,R_xv),
\end{array}\right.
\end{eqnarray}
here $\widetilde{L} h_2=-\mu^{-1/2}\left\{Q(\mu,\sqrt{\mu}h_2)+Q(\sqrt{\mu}h_2,\mu)\right\}.$ It is seen that $h_1+\sqrt{\mu}h_2$ is a solution to the initial boundary value problem \eqref{h.eq}.

It is convenient to consider the following problem before solving the system \eqref{h2}
\begin{eqnarray}\label{g}
\left\{\begin{array}{ll}
&\pa_t f+v\cdot\na_xf+\widetilde{L} f=0,\\[2mm]
&f(0,x,v)=f_0(x,v),\ \ f(t,x,v)|_-=f(t,x,R_xv).
\end{array}\right.
\end{eqnarray}
As a matter of fact, \eqref{g} can be treated by means of an $L^2-L^\infty$ argument, to do this, a weighted $L^\infty$ estimate should be established.
Let $l>6$, we denote
$\widetilde{f}=w_lf$, and study the equivalent linearized Boltzmann equation
\begin{eqnarray}\label{wg}
\left\{\begin{array}{ll}
&\pa_t \widetilde{f}+v\cdot\na_x\widetilde{f}+\widetilde{L}_w \widetilde{f}=0,\\[2mm]
&\widetilde{f}(0,x,v)=\widetilde{f}_0(x,v)=w_lf_0(x,v),\ \ \widetilde{f}(t,x,v)|_-=\widetilde{f}(t,x,R_xv),
\end{array}\right.
\end{eqnarray}
where $\widetilde{L}_w(\cdot)=w_l\widetilde{L}(\frac{\cdot}{w_l}).$

The solution $\widetilde{f}(t,x,v)$ to \eqref{wg} can be expressed through semigroup $U(t)$ as
\begin{equation*}
\widetilde{f}(t,x,v)=\{U(t)\widetilde{f}_0\}(x,v),
\end{equation*}
with the initial boundary data given by
$$
\{U(0)\widetilde{f}_0\}(x,v)=\widetilde{f}_0(x,v), \ \textrm{and} \ \{U(0)\widetilde{f}_0\}(x,v)|_{\ga_-}=\widetilde{f}_0(x,R_{x}v).
$$
As it is shown in \cite[pp.731]{Guo-2010}, the following conversation of mass, energy and angular momentum play a significant role in obtaining the time decay rate of \eqref{g}:
\begin{equation}\label{mass.c}
\int_{\Om\times\R^3}f(t,x,v)\sqrt{\mu(v)}dxdv=0,
\end{equation}
\begin{equation}\label{eng.c}
\int_{\Om\times\R^3}|v|^2f(t,x,v)\sqrt{\mu(v)}dxdv=0,
\end{equation}
\begin{equation}
\int_{\Omega \times \R^{3}}\{(x-x_{0})\times \varpi \}\cdot vf(t,x,v)
\sqrt{\mu }dxdv=0.  \label{axiscon2}
\end{equation}
Let us now use $\{e_1,e_2,e_3\}$ to denote the orthogonal basis of the $L^2_{x,v}$ space spanned by $\{1,(x-x_{0})\times \varpi \cdot v,v^2\}\sqrt{\mu}$
and define the $L_{x,v}^2$ projection $P$ as $P f=\sum\limits_{i=1}^3(f ,e_i)e_i,$ we also define $\{I-P\}f=f-Pf.$

The following result which has been proved in \cite[pp.777, Theorem 8]{Guo-2010} states the well-posedness and exponential decay of systems \eqref{g} and \eqref{wg}.
\begin{proposition}
\label{specularbd}
Assume that $\xi $ is
both strictly convex \eqref{scon} and analytic, and the mass \eqref{mass.c} and energy \eqref{eng.c} are conserved. In the case of $\Omega $ has
rotational symmetry \eqref{axis}, we also assume conservation of
corresponding angular momentum \eqref{axiscon2}. Let $\widetilde{f}_0=w_lf_{0}\in L^{\infty }$ with $l>6$. There exists a unique solution $f(t,x,v)$ to the system \eqref{g},
and $\widetilde{f}=U(t)\widetilde{f}_0$ to the system \eqref{wg}. Moreover, there exist
$\la_0>0$ and $C>0$ such that
\begin{equation*}\label{Ubd}
\left\|U(t)\widetilde{f}_0\right\|_{\infty}\leq Ce^{-\la_0t}\left\|\widetilde{f}_{0}\right\|_{\infty}.
\end{equation*}
\end{proposition}
With Proposition \ref{specularbd} in hand, we now provide the following lemma associated to the well-posedness of the system of \eqref{h2}
in weighted $L^\infty$ space.
\begin{lemma}\label{h2s}
Assume that $\xi $ is
both strictly convex \eqref{scon} and analytic, and $\Omega $ has
rotational symmetry \eqref{axis}.
Let $l>6$, assume $\sup\limits_t\|w_lg(t)\|_\infty<\infty$, there exists a unique global solution $h_2$ to the system
\begin{eqnarray}\label{h2g}
\left\{\begin{array}{ll}
&\pa_t h_2+v\cdot\na_xh_2+\widetilde{L} h_2=\mu^{-1/2}\chi_N Kg,\\[2mm]
&h_2(0,x,v)=0,\ \ h_2(t,x,v)|_-=h_2(t,x,R_xv).
\end{array}\right.
\end{eqnarray}
Moreover, if $P(h_2+\mu^{-1/2}g)=0$ and there exist $\eta_1>0$ and $\la_1>0$ such that $\|w_lg\|_\infty\leq \eta_1e^{-\la_1t}$, then there exists $0<\la_2<\la_1$
such that
\begin{equation}\label{deh2}\|w_lh_2\|_\infty\leq C\eta_1e^{-\la_2t}.
\end{equation}

\end{lemma}
\begin{proof}
Notice that  $\sup\limits_t\|w_lg(t)\|_\infty<\infty$, by Duhamel's principle and Proposition \ref{specularbd}, we see that there is indeed a unique solution $h_2$ to \eqref{h2g}, given by
\begin{equation*}\label{Uh2}
w_lh_2=\int_0^t U(t-s)w_l\mu^{-1/2}\chi_N [Kg](s)ds.
\end{equation*}
Next, if $P(h_2+\mu^{-1/2}g)=0$ and $\|w_lg\|_\infty\leq \eta_1e^{-\la_1t}$ with $l>6$, it is straightforward to see that
\begin{equation}\label{deph2}
\|w_l Ph_2\|_{\infty}\leq \|w_lP(\mu^{-1/2}g)\|_{\infty}\leq\sum\limits_{i=1}^3\|w_l(\mu^{-1/2}g,e_i)e_i\|_{\infty}
\leq C\|w_lg\|_\infty\leq C\eta_1e^{-\la_1t}.
\end{equation}
To prove \eqref{deh2}, it remains now to show the exponential decay rate for $\{I-P\}h_2,$ for this, we first act $\{I-P\}$ to \eqref{h2g} to obtain
\begin{eqnarray}\label{IPh2}
\left\{\begin{array}{ll}
&\pa_t \{I-P\}h_2+v\cdot\na_x\{I-P\}h_2+\widetilde{L} \{I-P\}h_2=\{I-P\}\left[\mu^{-1/2}\chi_N Kg\right],\\[2mm]
&\{I-P\}h_2(0,x,v)=0,\ \{I-P\}h_2(t,x,v)|_-=\{I-P\}h_2(t,x,R_xv),
\end{array}\right.
\end{eqnarray}
where we have used the fact that $\int_{\Omega\times \R^3}v\cdot\na_xh_2e_idxdv=0$ $(i\in\{1,2,3\})$, $v\cdot \na_xPh_2=0$ and $\{I-P\}\widetilde{L}h_2=\widetilde{L} \{I-P\}h_2$. Moreover, the fact that the rotational symmetry \eqref{axis} is also used to handle the boundary condition. Employing Proposition \ref{specularbd} again, we get from \eqref{IPh2}
\begin{equation*}\label{WIP}
w_l\{I-P\}h_2=\int_0^t U(t-s)w_l\{I-P\}\left\{\mu^{-1/2}\chi_N [Kg]\right\}(s)ds,
\end{equation*}
from which and Lemma \ref{op.es.lem} as well as the assumption on $g$, one has
\begin{equation}\label{WIP1}
\left\|w_l\{I-P\}h_2\right\|_{L^\infty}\leq C_N\eta_1\int_0^t e^{-\la_0(t-s)}e^{-\la_1s}ds.
\end{equation}
Let $\la_2=\frac{1}{2}\min\{\la_0,\la_1\}$, \eqref{WIP1} further yields
\begin{equation*}\label{WIP2}
\left\|w_l\{I-P\}h_2\right\|_{L^\infty}\leq C_N\eta_1e^{-\la_2t}.
\end{equation*}
This together with \eqref{deph2} implies \eqref{deh2}, thus the proof of Lemma \ref{h2s} is complete.

\end{proof}

We now go back to \eqref{h1} which is associated to the solvability of $h_1$. We first prove the following results.
\begin{lemma}\label{h1s}
There exist constants $\de_0>0$ and $\la_1>0$, such that if
\begin{equation}\label{h0s}
\|w_lh_0\|_{\infty}<\de_0\ \textrm{and}\ \sup\limits_{0\leq t\leq \infty}e^{\la_1 t}\|w_lg(t)\|_{\infty}<\de_0,\ \ l>6,
\end{equation}
then there exists a unique global solution
$h_1(t,x,v)$ to the initial boundary value problem
\begin{eqnarray}\label{h1g}
\left\{\begin{array}{ll}
&\pa_t h_1+v\cdot\na_xh_1+\nu h_1=\chi_N^c Kh_1+Q(h_1+g,h_1+g),\\[2mm]
&h_1(0,x,v)=h_0(x,v),\ \ h_1(t,x,v)|_-=h_1(t,x,R_xv),
\end{array}\right.
\end{eqnarray}
satisfying
\begin{equation}\label{h1decay}
\|w_lh_1(t)\|_{\infty}\leq Ce^{-\la_1 t}\|w_lh_0\|_{\infty},
\end{equation}
where $C>0$ and is independent of $\de_0.$

\end{lemma}

\begin{proof}
The existence of unique global solution of \eqref{h1g} will be proved via the following iteration scheme
\begin{eqnarray}\label{h1it}
\left\{\begin{array}{ll}
&\pa_t h^{n+1}_1+v\cdot\na_xh^{n+1}_1+\nu h^{n+1}_1=\chi_N^c Kh^{n+1}_1+Q(h^{n}_1+g,h^{n}_1+g),\\[2mm]
&h^{n+1}_1(0,x,v)=h_0(x,v),\ \ h^{n+1}_1(t,x,v)|_-=h^{n+1}_1(t,x,R_xv),
\end{array}\right.
\end{eqnarray}
starting with $h^{0}_1=h_0(x,v).$ In what follows, we first show that the approximation sequence $\{e^{\la_1t}w_lh_1^{n}\}_{n=0}^\infty$
is uniformly bounded, that is we will prove that if there exists $M_0>0$ such that $\sup\limits_{0\leq t\leq\infty}e^{\la_1t}\|w_lh_1^{n}(t)\|_\infty\leq M_0$ then $\sup\limits_{0\leq t\leq\infty}e^{\la_1t}\|w_lh_1^{n+1}(t)\|_\infty\leq M_0$. As for deriving \eqref{f.ep}, taking $(t,x,v)\notin [0,\infty)\times(\ga_0\cup\ga_-)$, by assuming $t_{m+1}\leq 0<t_m$ for some finite integer $m>0$,
we get from \eqref{h1it} that
\begin{equation}\label{h1n}
\begin{split}
w_lh^{n+1}_1(t,x,v)=&e^{-\nu(v) t}w_lh_0
+\sum\limits_{k=0}^{m-1}\int_{t_{k+1}}^{t_k}e^{-\nu(v)(t-s)}\left\{w_l\chi_N^c Kh^{n+1}_1\right\}(s,x_k+(s-t_k)v_k,v_k)ds\\
&+\int_{0}^{t_m}e^{-\nu(v)(t-s)}\left\{w_l\chi_N^c Kh^{n+1}_1\right\}(s,x_m+(s-t_m)v_m,v_m)ds
\\&+\sum\limits_{k=0}^{m-1}\int_{t_{k+1}}^{t_k}e^{-\nu(v)(t-s)}\left\{w_lQ(h^{n}_1+g,h^{n}_1+g)\right\}(s,x_k+(s-t_k)v_k,v_k)ds
\\&+\int_{0}^{t_m}e^{-\nu(v)(t-s)}\left\{w_lQ(h^{n}_1+g,h^{n}_1+g)\right\}(s,x_m+(s-t_m)v_m,v_m)ds.
\end{split}
\end{equation}
For $0<\la_1<\nu_0/2$, \eqref{h1n} further implies
\begin{equation}\label{h1bd}
\begin{split}
e^{\la_1t}&\|w_lh^{n+1}_1(t)\|_{_{\infty}}\\ \leq &\|w_lh_0\|_{_{\infty}}
+e^{\la_1t}\left\{\sum\limits_{k=0}^{m-1}\int_{t_{k+1}}^{t_k}+\int_{0}^{t_m}\right\}e^{-\frac{\nu_0}{2}(t-s)}e^{-\frac{\nu(v)}{2}(t-s)}
\left\{\left[\frac{C}{N}+\eps\right](1+|v|)\|w_lh^{n+1}_1\|_{_{\infty}}(s)\right\}ds
\\&+Ce^{\la_1t}\left\{\sum\limits_{k=0}^{m-1}\int_{t_{k+1}}^{t_k}+\int_{0}^{t_m}\right\}e^{-\frac{\nu_0}{2}(t-s)}e^{-\frac{\nu(v)}{2}(t-s)}
(1+|v|)\left\{\|w_lh^{n}_1\|^2_{_{\infty}}(s)+\de_0\|w_lh^{n}_1\|_{_{\infty}}(s)
\right\}ds
\\&+C\de_0^2e^{\la_1t}\left\{\sum\limits_{k=0}^{m-1}\int_{t_{k+1}}^{t_k}+\int_{0}^{t_m}\right\}
e^{-\frac{\nu_0}{2}(t-s)}e^{-\frac{\nu(v)}{2}(t-s)}e^{-2\la_1s}
(1+|v|)ds\\
\leq &\|w_lh_0\|_{_{\infty}}
+\left\{\sum\limits_{k=0}^{m-1}\int_{t_{k+1}}^{t_k}+\int_{0}^{t_m}\right\}e^{\la_1s}e^{-\frac{\nu(v)}{2}(t-s)}
\left\{\left[\frac{C}{N}+\eps\right](1+|v|)\|w_lh^{n+1}_1\|_{_{\infty}}(s)\right\}ds
\\&+C\left\{\sum\limits_{k=0}^{m-1}\int_{t_{k+1}}^{t_k}+\int_{0}^{t_m}\right\}e^{\la_1s}e^{-\frac{\nu(v)}{2}(t-s)}
(1+|v|)\left\{\|w_lh^{n}_1\|^2_{_{\infty}}(s)+\de_0\|w_lh^{n}_1\|_{_{\infty}}(s)
\right\}ds\\&+C\de_0^2\left\{\sum\limits_{k=0}^{m-1}\int_{t_{k+1}}^{t_k}+\int_{0}^{t_m}\right\}e^{-\frac{\nu(v)}{2}(t-s)}e^{-\la_1s}
(1+|v|)ds,
\end{split}
\end{equation}
according to Lemma \ref{op.es.lem} and \eqref{h0s}. Choosing $N>0$ large enough and $\eps>0$ suitably small, we then have from \eqref{h1bd}
that
\begin{equation}\label{h1bd2}
\begin{split}
\sup\limits_{0\leq t\leq\infty}e^{\la_1t}\|w_lh^{n+1}_1(t)\|_{_{\infty}}\leq C\de_0+ C(\de_0^2+\de_0M_0+M_0^2).
\end{split}
\end{equation}
Therefore $\sup\limits_{0\leq t\leq\infty}e^{\la_1t}\|w_lh^{n+1}_1\|_{_{\infty}}\leq M_0$ follows if both $\de_0>0$ and $M_0>0$ are sufficiently small.

Next we prove that $\{e^{\la_1t}w_lh_1^{n}\}_{n=0}^\infty$ is a Cauchy sequence in $L^{\infty}-$norm. For this, from the difference of system \eqref{h1n} for $n+1$ and $n$, it follows for $(t,x,v)\notin [0,\infty)\times(\ga_0\cup\ga_-)$
\begin{equation*}\label{h1nd}
\begin{split}
&w_l(h^{n+1}_1-h^{n}_1)(t,x,v)\\
&\quad=\sum\limits_{k=0}^{m-1}\int_{t_{k+1}}^{t_k}e^{-\nu(t-s)}\left\{w_l\chi_N^c K(h^{n+1}_1-h^{n}_1)\right\}(s,x_k+(s-t_k)v_k,v_k)ds
\\&\qquad+\int_{0}^{t_m}e^{-\nu(t-s)}\left\{w_l\chi_N^c K(h^{n+1}_1-h^{n}_1)\right\}(s,x_m+(s-t_m)v_m,v_m)ds
\\&\qquad+\sum\limits_{k=0}^{m-1}\int_{t_{k+1}}^{t_k}e^{-\nu(t-s)}\left\{w_l\left[Q(h^{n}_1+g,h^{n}_1+g)-Q(h^{n-1}_1+g,h^{n-1}_1+g)\right]\right\}
(s,x_k+(s-t_k)v_k,v_k)ds
\\&\qquad+\int_{0}^{t_m}e^{-\nu(t-s)}\left\{w_l\left[Q(h^{n}_1+g,h^{n}_1+g)-Q(h^{n-1}_1+g,h^{n-1}_1+g)\right]\right\}
(s,x_m+(s-t_m)v_m,v_m)ds.
\end{split}
\end{equation*}
Notice that
$$
Q(h^{n}_1+g,h^{n}_1+g)-Q(h^{n-1}_1+g,h^{n-1}_1+g)=Q(h^{n}_1-h^{n-1}_1,h^{n}_1+g)+Q(h^{n-1}_1+g,h^{n}_1-h^{n-1}_1).
$$
Applying Lemma \ref{op.es.lem} and \eqref{h0s} again, we have
\begin{equation*}\label{h1nd2}
\begin{split}
\sup\limits_{0\leq t\leq\infty}e^{\la_1t}\|w_l(h^{n+1}_1-h^{n}_1)(t)\|_\infty\leq C(M_0+\de_0)\sup\limits_{0\leq t\leq\infty}e^{\la_1s}\|w_l(h^{n}_1-h^{n-1}_1)(t)\|_\infty,
\end{split}
\end{equation*}
where the fact that
$\sup\limits_{0\leq t\leq\infty}e^{\la_1t}\|w_lh_1^{n}(t)\|_\infty\leq M_0$ was also used.
Therefore $\{e^{\la_1t}w_lh_1^{n}\}_{n=0}^\infty$ is a Cauchy sequence in $L^\infty-$norm, the limit $h_1$ is a desired solution of \eqref{h1g}. Moreover, letting $M_0\leq \|w_lh_0\|_\infty$ and taking limit in \eqref{h1bd2} leads us to \eqref{h1decay}, this ends the proof of Lemma \ref{h1s}.

\end{proof}

Once Lemmas \ref{h2s} and \ref{h1s} are obtained, we can now complete
\begin{proof}
[The proof of Proposition \ref{h.exist}]
Recall \eqref{h.eq} is equivalent to \eqref{h1} and \eqref{h2}. We approximate the system of equations \eqref{h1} and \eqref{h2}
as follows:
\begin{eqnarray}\label{h11}
\left\{\begin{array}{ll}
&\pa_t h^{n+1}_1+v\cdot\na_xh^{n+1}_1+\nu h^{n+1}_1=\chi_N^c Kh^{n+1}_1+Q(h^{n+1}_1+\sqrt{\mu}h^n_2,h^{n+1}_1+\sqrt{\mu}h^n_2),\\[2mm]
&h^{n+1}_1(0,x,v)=h_0(x,v),\ \ h^{n+1}_1(t,x,v)|_-=h^{n+1}_1(t,x,R_xv),
\end{array}\right.
\end{eqnarray}
\begin{eqnarray}\label{h22}
\left\{\begin{array}{ll}
&\pa_t h^{n+1}_2+v\cdot\na_xh^{n+1}_2+\widetilde{L} h^{n+1}_2=\mu^{-1/2}\chi_N Kh^{n+1}_1,\\[2mm]
&h^{n+1}_2(0,x,v)=0,\ \ h^{n+1}_2(t,x,v)|_-=h^{n+1}_2(t,x,R_xv),
\end{array}\right.
\end{eqnarray}
and start with $h^{0}_1=h_0(x,v)$ and $h^0_2=0.$
The summation of \eqref{h11} and $\sqrt{\mu}\times\eqref{h22}$ yields
\begin{eqnarray}\label{h1ad2}
\left\{\begin{array}{ll}
&\pa_t \left\{h^{n+1}_1+\sqrt{\mu}h^{n+1}_2\right\}+v\cdot\na_x\left\{h^{n+1}_1+\sqrt{\mu}h^{n+1}_2\right\}+L h^{n+1}_1\\[2mm]
&\qquad\qquad\qquad=Q(h^{n+1}_1+\sqrt{\mu}h^n_2,h^{n+1}_1+\sqrt{\mu}h^n_2),\\[2mm]
&\left\{h^{n+1}_1+\sqrt{\mu}h^{n+1}_2\right\}(0,x,v)=h_0(x,v),\ \ \left\{h^{n+1}_1+\sqrt{\mu}h^{n+1}_2\right\}|_-=\left\{h^{n+1}_1+\sqrt{\mu}h^{n+1}_2\right\}(t,x,R_xv).
\end{array}\right.
\end{eqnarray}
Recall $F_0(x,v)=\mu+h_0(x,v)$, one sees that
$\mu^{-1/2}h_0(x,v)$ satisfies
\eqref{mass.c}, \eqref{eng.c} and
\eqref{axiscon2}, therefore, from \eqref{h1ad2}, it follows $P(h^{n+1}_2+\mu^{-1/2}h^{n+1}_1)=0$ for all $n\geq0$.
Furthermore, let us suppose $0<\eta_1\leq \|w_lh_0\|_\infty$ in Lemma \ref{h1s}, we first construct $h_1^1$ via Lemma \ref{h1s} with $g=h_2^0=0$, which has the require exponential decay \eqref{h1decay}, and then determine $h_2^1$ through Lemma \ref{h2s} with $g=h_1^1$.
then applying Lemmas \ref{h2s} and \ref{h1s} repeatedly and by means of an induction over $n$, we
deduce that the sequences $\{h^{n}_1\}_{n=0}^{\infty}$ and $\{h^{n}_2\}_{n=0}^{\infty}$ are well-defined and enjoy the following exponential decay estimates
\begin{equation}\label{h12ndec}
\|w_lh^{n}_1(t)\|_{\infty}\leq Ce^{-\la t}\|w_lh_0\|_{\infty},\ \|w_lh^{n}_2(t)\|_{\infty}\leq Ce^{-\la t}\|w_lh_0\|_{\infty},
\end{equation}
where $\la=\min\{\la_1,\la_2\}$ and $C$ is independent of $n$.

Next, we regard exactly the same iterative scheme for $h_1^{n+1}$
as in the proof of Lemma \ref{h1s} with $g$ replaced by $h^n_2$ . Since the bounds \eqref{h12ndec} are uniform with respect to $n$, we are able to
derive the same estimates as in the latter proof independently of $h_2^n$. Namely, $\{h_1^{n}\}_{n=0}^\infty$ can be proved
to be a Cauchy sequence in $L^\infty-$norm
and therefore converges
strongly towards a function $h_1$.
By \eqref{h12ndec} again, the sequence $\{h_2^{n}\}_{n=0}^\infty$
is uniformly bounded in $L^\infty-$norm
and thus converges, up to a subsequence, weakly-* toward a function $h_2$.

Furthermore, taking the weak limit inside the iterative
scheme, we see that $(h_1, h_2)$ is a solution to the system \eqref{h1} and \eqref{h2}, which implies
$h = h_1 + \sqrt{\mu}h_2$ is a solution to the perturbed system \eqref{h.eq}. In addition,
taking the limit inside the exponential decays \eqref{h12ndec} leads us to the expected
exponential decay  \eqref{hdecay} for $h$. Finally, the uniqueness, positivity and continuity of the solution follows from the same argument as that of \cite[pp.62-68]{BG-2015}. This completes the proof of Proposition \ref{h.exist}.

\end{proof}

\section{The proof of Theorem \ref{mth}}\label{proof}
In this final section, we make use of Propositions \ref{loc.ex}, \ref{h.exist} and \ref{Gex} to complete
\begin{proof}[The proof of Theorem \ref{mth}]
Thanks to Proposition \ref{Gex}, given $\de_0>0$, there exists $t_{\ast\ast}>0$ such that
\begin{equation}\label{loct1}
\|w_{l_0}[G(t_{\ast\ast},v)-\mu]\|_\infty\leq \de_0/2.
\end{equation}
For such two constants $\de_0>0$ and $t_{\ast\ast}>0$,
if $\|w_{l_0}f_0\|$ is small enough, by Proposition \ref{loc.ex}, there exists a unique solution $F(t,x,v)$ to \eqref{be} such that
\begin{equation}\label{loct2}
\|w_{l_0}[F(t_{\ast\ast},x,v)-G(t_{\ast\ast},v)]\|_\infty\leq \de_0/2.
\end{equation}
\eqref{loct1} and \eqref{loct2} give rise to
\begin{equation*}\label{loct3}
\|w_{l_0}[F(t_{\ast\ast},x,v)-\mu(v)]\|_\infty\leq \de_0.
\end{equation*}
Now,
by Proposition \ref{h.exist} with initial time $t_{\ast\ast}$ and initial datum $h_0=F(t_{\ast\ast})-\mu$, we see that $F(t,x,v)\geq0$ exists
uniquely also for $t>t_{\ast\ast}$ and converges to $\mu$ as $t\rightarrow\infty$. This finishes the proof of Theorem \ref{mth}.
\end{proof}

\medskip

\noindent {\bf Acknowledgements:} YG was supported in part by NSFC grant 10828103, NSF grant 1209437, Simon Research
Fellowship and BICMR. SQL was
supported by grants from the National Natural Science Foundation of China (contracts: 11471142, 11271160 and 11571063) and China
Scholarship Council.
SQL would like to express his gratitude for the generous hospitality of the Division of Applied Mathematics at Brown University during his visit.


\end{document}